\documentclass[english]{article}
\usepackage[T1]{fontenc}
\usepackage[latin9]{inputenc}
\usepackage{amsmath}
\usepackage{amsthm}
\usepackage{amssymb}

\makeatletter

\numberwithin{equation}{section}
\theoremstyle{plain}
\newtheorem{thm}{\protect\theoremname}[section]
\theoremstyle{plain}
\newtheorem{prop}{\protect\propositionname}[section]
\theoremstyle{plain}
\newtheorem{lem}{\protect\lemmaname}[section]
\ifx\proof\undefined\
  \newenvironment{proof}[1][\proofname]{\par
    \normalfont\topsep6\p@\@plus6\p@\relax
    \trivlist
    \itemindent\parindent
    \item[\hskip\labelsep
          \scshape
      #1]\ignorespaces
  }{%
    \endtrivlist\@endpefalse
  }
  \providecommand{\proofname}{Proof}
\fi
\theoremstyle{remark}
\newtheorem{rem}{\protect\remarkname}[section]
\theoremstyle{plain}
\newtheorem{defn}{\protect\definitionname}[section]
\theoremstyle{plain}
\newtheorem{proposition}{\protect\propositionname}[section]

\usepackage{graphicx}
\usepackage{float}
\usepackage{comment}

\usepackage{hyperref}
\hypersetup{
     colorlinks   = true,
     citecolor    = blue,
     linkcolor    = blue
}

\date{}

\makeatother

\usepackage{babel}
\providecommand{\lemmaname}{Lemma}
\providecommand{\propositionname}{Proposition}
\providecommand{\remarkname}{Remark}
\providecommand{\theoremname}{Theorem}
\providecommand{\definitionname}{Definition}
\providecommand{\propositionname}{Proposition}

\begin{document}
\title{Parabolic Anderson Model in Hyperbolic Spaces and Phase Transition}


\author{Xi Geng\thanks{School of Mathematics and Statistics, The University of Melbourne, Melbourne, VIC 3010. Email: xi.geng@unimelb.edu.au. XG gratefully acknowledges the support by ARC grant DE210101352.} \ 
and Cheng Ouyang \thanks{Dept. Mathematics, Statistics and Computer Science, University of Illinois at Chicago, Chicago, IL 60607. Email: couyang@uic.edu. }
}

\maketitle

\begin{abstract}
    Consider a Parabolic Anderson model (PAM) with Gaussian noise that is white in time and colored in space, where the spatial correlation decays polynomially with order $\alpha$. In Euclidean spaces with dimension greater than $2$, it is well-understood that the critical value for $\alpha$ is $2$. Specifically, for $\alpha<2$, the second moment of the solution grows exponentially over time, while for $\alpha>2$, there is a phase transition, from the second moment being uniformly bounded in time to exhibiting exponential growth in time when the inverse temperature increases. This critical behavior arises from the fact that in Euclidean space, Brownian motion tends to infinity at a speed of $\sqrt{t}$.

The present work explores the PAM on a hyperbolic space. Given that Brownian motion in a hyperbolic space travels at a speed of $t$, one expects that $\alpha=1$ would be the critical value for the above phenomena. We confirm that this intuition is indeed correct. Furthermore, we uncover a novel phase for $\alpha<1$ in which the second moment explodes sub-exponentially, distinct from the behavior observed in Euclidean space.

\end{abstract}

\tableofcontents

\section{Introduction and main results}

Let $M\triangleq\mathbb{H}^{d}$ denote the $d$-dimensional space-form
of curvature $K\equiv-1$. We consider the parabolic Anderson model (PAM) on $M$
\begin{equation}
\begin{cases}
\partial_{t}u(t,x)=\Delta u(t,x)+\beta u(t,x)\cdot \frac{\partial^2}{\partial t\partial x}{W}(t,x),\\
u(0,\cdot)\equiv1.
\end{cases}\label{eq:HypPAM}
\end{equation}
in the sense It\^{o}-Walsh \cite{Wa86}. Here $\Delta$ is the hyperbolic Laplacian, $\beta>0$ is the (inverse) temperature parameter, and
$W$ is a time-dependent Gaussian field on $M$ that is white in time and
colored in space with spatial covariance function $f:M\times M\rightarrow\mathbb{R}$,
i.e. 
\[
\mathbf{E}\big[W(s,x)W(t,y)\big]=(s\wedge t)f(x,y).
\]
Throughout our discussion, we assume that $f$ is uniformly bounded and vanishes when the (hyperbolic) distance $\rho(x,y)$ tends to infinity. 

Standard Wiener chaos expansion technique shows that the stochastic differential equation \eqref{eq:HypPAM} has a unique solution which can be written in the mild form as the solution to the following stochastic integral equation (see, e.g., \cite{BCO25}):
\begin{equation}\label{HypPAM-Mild}
u(t,x)=1+\int_0^t\int_Mp_{t-s}(x,y)u(s,y)W(ds,dy),
\end{equation}
where $p_t(x,y)$ is the heat kernel on $M$, and the stochastic integral is understood in the sense of It\^{o}-Walsh \cite{Wa86}. 

The parabolic Anderson model arises in a large number of diverse questions in probability theory and mathematical physics. For example, it gives rise to the free energy of the  directed polymer and to the Cole-Hopf solution of the KPZ equation \cite{ACQ11, Ka87, KPZ86}; it also has direct connections  with the stochastic Burger's equation \cite{CM94} and Majda's model of shear-layer flow in turbulent diffusion \cite{Mj93}.  In the last few years, there have been many advances in understanding numerous fundamental properties of the PAM on flat spaces (e.g. $M=\mathbb{R}^d$ and $M=\mathbb{T}^d)$, such as intermittency property (see, e.g., \cite{Ch15}, \cite{Kh14} and references tehrein), energy landscape and fluctuation \cite{BQS11, DGK23}. 

In particular, suppose $M=\mathbb{R}^d$ for $d\geq 3$ and the covariance function satisfies
\begin{align}\label{covar-Euclidean}
    f(x,y)=f(x-y)\asymp \frac{1}{|x-y|^\alpha}
\end{align}
for some $\alpha>0$ and all large $|x-y|$. The second moment of the solution $u$ exhibits the following large-time behavior (see, e.g., \cite{CK19}):
\begin{itemize}
    \item [\bf E-(i)] When $\alpha<2$, the second moment $\mathbf{E}\big[u(t,x)^2\big]$ blows up exponentially in time,
    $$\log\mathbf{E}\big[u(t,x)^2\big]\asymp t.$$
    
    \item [\bf E-(ii)] When $\alpha>2$, there is phase transition, that is, for all large $\beta$ one has 
     $$\log\mathbf{E}\big[u(t,x)^2\big]\asymp t;$$
     whereas  $\mathbf{E}\big[u(t,x)^2\big]$ is uniformly bounded in $t$ for all small $\beta$.
\end{itemize}
Here is a heuristic explanation for why $\alpha=2$ is critical. Feynman-Kac formula for the second moment (\cite{HNS11}) states that
\begin{align}\label{FK-moment-intro}
\mathbf{E}\big[u(t,x)^2\big]=\mathbb{E}\Big[\exp\Big(\beta^2\int_0^tf(B_s-\tilde{B}_s)ds\Big)\Big].
\end{align}
In the above, the expectation $\mathbb{E}$ is taken with respect to the randomness of two independent Brownian motions $B$ and $\tilde{B}$ starting from the same point $x$. Note that $B_s-\tilde{B}_s$ is a Euclidean Brownian motion; when $d\geqslant3$, it is transient and travels to infinity at a speed of $O(\sqrt{s})$. Hence, if one naively replaces $B_s-\tilde{B}_s$ by $\sqrt{s}$ for large $s$ in \eqref{FK-moment-intro} and employ \eqref{covar-Euclidean}, it is clear that $\alpha=2$ is a critical value: when $\alpha>2$ the time integral is uniformly bounded for all $t$, while when $\alpha<2$ the time integral is infinity when $t$ tends to infinity.

In the hyperbolic space $\mathbb{H}^d$, the Brownian motion is transient in all dimensions and goes to infinity at a speed of $O(t)$. The same heuristic argument above suggests $\alpha=1$ is critical in a hyperbolic space. Naturally, one asks whether this argument leads to the correct answer. Our analysis in this paper gives an affirmative answer to this question, and can be summarized in the following theorems.

\begin{thm}\label{thm1-intro} Let $\rho(x,y)$ be the hyperbolic distance between $x$ and $y$ on $\mathbb{H}^d$.
\begin{itemize}
    \item [(a)]
Suppose $\alpha>1$ and the covariance function $f$ satisfies
the following assumption:
\begin{equation}
\big|f(x,y)\big|\leqslant\frac{C}{\rho(x,y)^{\alpha}}\ \ \ \forall x,y:\rho(x,y)>R\label{eq:UpDecayAssump}
\end{equation}
with some given constants $C,R>0$. Then there exists
a constant $\beta_{0}>0$, such that 
\[
\sup_{\substack{t\geqslant0,x\in M}
}\mathbf{E}\big[u(t,x)^{2}\big]<\infty
\]
for all $\beta\in(0,\beta_{0})$.

\item [(b)]
On the other hand, suppose $\alpha<1$ and $f$ satisfies: 
\begin{equation}
f(x,y)\geqslant\frac{C}{\rho(x,y)^{\alpha}}\ \ \ \forall x,y:\rho(x,y)\geqslant R\label{eq:LowDecayCov}
\end{equation}
with suitable constants $C,R>0$. Then 
$\mathbf{E}[u(t,x)^2]$ blows up as $t\uparrow\infty$ for all $\beta>0$.
\end{itemize}
\end{thm}

Similarly (to the Euclidean setting), when the inverse temperature $\beta$ is large enough, the second moment blows up exponentially independent of the choice of $\alpha.$

\begin{thm}\label{thm2-intro}
    For any $\alpha>0$, there exists $\beta_1>0$ depending on $\alpha$, such that for all $\beta>\beta_1$ one has
    \begin{equation}\label{eq:LargeBLower}
\underset{t\rightarrow\infty}{\underline{\lim}}\frac{1}{t}\log\mathbf{E}[u(t,x)^{2}]>0.
\end{equation}
    
\end{thm}
It is clear from the above two theorems that the second moment of the PAM presents a similar large-time behavior to the Euclidean case as described in {\bf E-(i)} and {\bf E-(ii)}, and, as expected, the critical value for $\alpha$ is $1$ because a hyperbolic Brownian motion travels at a speed of $O(t)$ as oppose to $O(\sqrt{t})$ of a Euclidean Brownian motion. However, our analysis discovers a novel phase transition that is absent in the Euclidean situation for $\alpha<1.$

\begin{thm}\label{thm3-intro}
    Let $\alpha<1$ and assume that  as   $\rho(x,y)$ tends to infinity,
    \begin{align}f(x,y)\asymp \frac{1}{\rho(x,y)^{\alpha}}.\label{decay-covari1-intro}\end{align}
    Then there exists a $\bar{\beta}_0$ such that for all $\beta<\bar{\beta}_0$, one has
    $$\log\mathbf{E}[u(t,x)^{2}]\asymp {t^{1-\alpha}},\quad\mathrm{as}\ t\uparrow\infty.$$
    
\end{thm}
Theorems \ref{thm2-intro} and \ref{thm3-intro} together show that, unlike the Euclidean case described in {\bf E-(ii)}, when $\alpha<1$ the hyperbolic PAM undergoes a subtle phase transition in terms of the blow-up rate of its second moment: the second moment blows up sub-exponentially when $\beta$ is small, whereas it blows up exponentially when $\beta$ is large. As one will see below, this new phenomenon arises because when the noise intensity is low (small $\beta$), the PAM has to pick up the noise from far away to grow. Since hyperbolic Brownian motions travel to infinity at a much faster speed than a Euclidean Brownian motion, this together with the decay of the covariance function \eqref{decay-covari1-intro} gives a slower blow-up rate.  On the other hand, when the noise intensity is high (large $\beta$), the local environment is already strong enough to cause an exponential blow-up for the moment.  

The significance of the large-time behavior of the second moment can be tracked back to the so-called {\it $L^2$-region} for directed polymers on $\mathbb{Z}^d$, in which the second moment of the (normalized) partition function can be computed rather explicitly (see, e.g., \cite{Co16}). In the $L^2$-region, the partition function converges to a positive random variable almost surely and the polymer exhibits diffusive behavior under the polymer measure.  The parabolic Anderson model is closely connected to the theory of continuum directed polymers, as both share the same probability law through the polymer partition function. Consequently, analysis of the second moment of the PAM yields important information about the statistical properties of continuum polymers. The new phase transition described in the last paragraph for hyperbolic PAM raises the natural question of whether this transition corresponds to new qualitative behaviors of the polymer that are absent in Euclidean geometries. This is a reasonable expectation, given the distinct properties of hyperbolic Brownian motion compared to its Euclidean counterpart. These differences suggest the possibility of novel localization phenomena or fluctuation regimes for the associated polymer measure. Exploring these potential behaviors in the hyperbolic setting will be a focus of our future work.

To conclude the introduction of our main results, we would like to make a comment on the covariance function on hyperbolic spaces with polynomial decay when $\rho(x,y)\uparrow\infty$,
\begin{align}\label{covar-poly-intro}f(x,y)\asymp\frac{1}{\rho(x,y)^\alpha}\end{align} for $\alpha>0$.  Due to the geometry of hyperbolic spaces, it is not a priori clear whether such a function exists. Indeed, most known positive definite functions on hyperbolic spaces decay exponentially fast as $\rho(x,y)\uparrow\infty$ (see, e.g., \cite{BCO25, BHV09}). In Section \ref{sec: covar-poly} below, we give an explicit construction of positive definite functions satisfying \eqref{covar-poly-intro}. As one will see, this construction is non-trivial and may be of independent interest.

\begin{rem}
    In a recent work \cite{BCO25}, moment estimates have been studied for the parabolic Anderson model on Cartan-Hadamard manifolds. The covariance function of the noise considered in this work decays exponentially fast as $\rho(x,y)\uparrow\infty$ and falls under the setting of Theorem \ref{thm1-intro}-(a). Therefore,  the transitions discussed in the present paper do not appear in \cite{BCO25}.
\end{rem}


\section{Upper estimates}

In this section, we prove Part (a) of Theorem \ref{thm1-intro}. We assume that the covariance function $f(x,y)$ satisfies
\[
|f(x,y)|\leqslant\frac{C}{\rho(x,y)^{\alpha}}
\]
for all $x,y$ with $\rho(x,y)>R$, where $\alpha>1$ and $C,R$ are
given positive constants. 
\subsection{Some basic representations}

The analysis of $u(t,x)$ is largely based on its Feynman-Kac representation
stated below. Its proof is standard and is thus omitted.
\begin{prop}
\label{prop:FKRep}The solution $u(t,x)$ to the PAM (\ref{eq:HypPAM}) is given by 
\begin{equation}
u(t,x)=\mathbb{E}_{x}\big[\exp\big(\beta\int_{0}^{t}W(ds,B_{s})-\frac{1}{2}\beta^{2}\int_{0}^{t}f(B_{s},B_{s})ds\big)\big],\label{eq:FKRep}
\end{equation}
where $B$ is a hyperbolic Brownian motion starting at $x$ that is independent
of $W$ and $\mathbb{E}_{x}$ means taking expectation with respect
to $B$. 
\end{prop}
An important consequence of Proposition \ref{prop:FKRep} is the following
representation of the second moment of $u(t,x)$.
\begin{lem}
\label{lem:2MomRep}One has 
\begin{equation}
\mathbf{E}\big[u(t,x)^{2}\big]=\mathbb{E}_{x,x}\big[\exp\big(\beta^{2}\int_{0}^{t}f(B_{s},\tilde{B}_{s})ds\big)\big].\label{eq:2MomRep}
\end{equation}
The expectation $\mathbf{E}$ on the left hand side is taken with respect
to the Gaussian field $W$. On the right hand side, $B,\tilde{B}$ are two independent
copies of hyperbolic Brownian motions starting at $x$ (jointly independent of $W$)
and $\mathbb{E}_{x,x}$ means taking expectation with respect to them.
\end{lem}
\begin{proof}
By using the Feynman-Kac representation (\ref{eq:FKRep}), one can
write
\[
u(t,x)^{2}=\mathbb{E}_{x,x}\big[\exp\big(\beta N_{t}(B)+\beta N_{t}(\tilde{B})-\frac{1}{2}\beta^{2}\int_{0}^{t}\big(f(B_{s},B_{s})+f(\tilde{B}_{s},\tilde{B}_{s})\big)ds\big)\big],
\]
where we set $N_{t}(\omega)\triangleq\int_{0}^{t}W(ds,\omega_{s})$
for each fixed path $\omega$ in $M$. Note that $N_{t}(\omega)$
is a Gaussian martingale with quadratic variation $\int_{0}^{t}f(\omega_{s},\omega_{s})ds$
(under $\mathbf{E}$). In addition, given two fixed paths $\omega^{1},\omega^{2}$
in $M$, the process $N_{t}(\omega^{1})+N_{t}(\omega^{2})$ is a Gaussian
martingale with quadratic variation 
\begin{equation}
V_{t}(\omega^{1},\omega^{2})\triangleq\int_{0}^{t}f(w_{s}^{1},w_{s}^{1})ds+\int_{0}^{t}f(w_{s}^{2},w_{s}^{2})ds+2\int_{0}^{t}f(w_{s}^{1},w_{s}^{2})ds.\label{eq:2MomRepPf1}
\end{equation}
In particular, one has 
\begin{equation}
\mathbf{E}\big[\exp\big(\beta N_{t}(\omega^{1})+\beta N_{t}(\omega^{2})-\frac{1}{2}\beta^{2}V_{t}(\omega^{1},\omega^{2})\big)\big]=1.\label{eq:2MomRepPf2}
\end{equation}
It follows from (\ref{eq:2MomRepPf1}, \ref{eq:2MomRepPf2}) that
\begin{align*}
 & \mathbf{E}\big[u(t,x)^{2}\big]\\
 & =\mathbf{E}\mathbb{E}_{x,x}\Big[\exp\big(\beta N_{t}(B)+\beta N_{t}(\tilde{B})-\frac{1}{2}\beta^{2}V_{t}(B,\tilde{B})+\beta^{2}\int_{0}^{t}f(B_{s},\tilde{B}_{s})ds\big)\Big]\\
 & =\mathbb{E}_{x,x}\Big[\mathbf{E}\big[\exp\big(\beta N_{t}(B)+\beta N_{t}(\tilde{B})-\frac{1}{2}\beta^{2}V_{t}(B,\tilde{B})\big]\exp\big(\beta^{2}\int_{0}^{t}f(B_{s},\tilde{B}_{s})ds\big)\Big]\\
 & =\mathbb{E}_{x,x}\Big[\exp\big(\beta^{2}\int_{0}^{t}f(B_{s},\tilde{B}_{s})ds\big)\Big].
\end{align*}
This gives the desired relation (\ref{eq:2MomRep}).
\end{proof}

\subsection{Some hyperbolic estimates}

We recall two fundamental estimates in hyperbolic geometry. The first
one is about uniform estimates for the heat kernel.

\begin{lem}\label{lem:HKEst}(\cite[Theorem 3.1]{DM98}) Let $H_{t}(\rho(x,y))$
be the heat kernel on $M$. Then one has 
\begin{equation}
H_{t}(\rho)\asymp t^{-d/2}\exp\big(-\frac{(d-1)^{2}t}{4}-\frac{\rho^{2}}{4t}-\frac{(d-1)\rho}{2}\big)(1+\rho+t)^{\frac{d-3}{2}}(1+\rho)\label{eq:HKEst}
\end{equation}
uniformly for $\rho\geqslant0$ and $t>0$. Here ``$A\asymp B$''
means $c_{1}B\leqslant A\leqslant c_{2}B$ with some universal constants
$c_{1},c_{2}>0.$

\end{lem}

The second estimate is a reversed triangle inequality that is only
valid in hyperbolic geometry. 

\begin{lem}{\label{lem:RevTri}}(\cite[Lemma 3.4]{HL10}) Let $\Delta ABC$
be a hyperbolic triangle with geodesic edges $a,b,c$. Then one has
\begin{equation}
0\leqslant b+c-a\leqslant\log\frac{2}{1-\cos A}.\label{eq:RevTri}
\end{equation}

\end{lem}
\begin{rem}
Lemma \ref{lem:HKEst} shows that the Laplacian has a strictly positive
bottom spectrum $\lambda_{0}=(d-1)^{2}/4$ and the heat kernel decays
like $t^{-3/2}e^{-\lambda_{0}t}$ as $t\rightarrow\infty.$ A remarkable
point about Lemma \ref{lem:RevTri} is that the estimate is uniform
in $a,b,c$ (with the angle $A$ fixed). In particular, when
$b,c$ are very large (so is $a$ as a function of $b,c$ when $A$
is fixed), taking the detour path $(b,c)$ to go from $B$ to $C$
is not-so-different from directly taking the geodesic path $a$. These two properties
are both consequences of negative curvature and are drastically different
from the Euclidean situation. Their implications on Brownian motion are contained
in the estimates (\ref{eq:LowBdDist1}, \ref{eq:FlucEst}) below. 
\end{rem}

\subsection{Reduction to a single integral estimate}

To prove Theorem \ref{thm1-intro} (a), we begin by applying triangle
inequality to the Feynman-Kac representation (\ref{eq:2MomRep}) so that 
\begin{equation}
\mathbf{E}\big[u(t,x)^{2}\big]\leqslant\mathbb{E}_{x,x}\Big[\exp\big(\beta^{2}\int_{0}^{\infty}|f|(B_{t},\tilde{B}_{t})dt\big)\Big]=:I(\beta).\label{eq:IBet3}
\end{equation}
To ease notation, we will assume $f\geqslant0$ (otherwise just regard $|f|$
as $f$). One has 
\begin{align*}
\exp\Big(\beta^{2}\int_{0}^{\infty}f(B_{t},\tilde{B}_{t})dt\Big) & =\sum_{n=0}^{\infty}\frac{\beta^{2n}}{n!}\Big(\int_{0}^{\infty}f(B_{t},\tilde{B}_{t})dt\Big)^{n}\\
 & =\sum_{n=0}^{\infty}\beta^{2n}\int_{0<t_{1}<\cdots<t_{n}<\infty}\prod_{k=1}^{n}f(B_{t_{k}},\tilde{B}_{t_{k}})dt_{1}\cdots dt_{n}.
\end{align*}
Note that $(B_{t},\tilde{B}_{t})$ can be viewed as the Brownian motion on the
product space $M\times M$ (equipped with the product metric), whose
heat kernel is clearly given by 
\[
P_{t}((x,y),(x',y'))=H_{t}(\rho(x,x'))H_{t}(\rho(y,y')).
\]
It follows that 
\begin{align}
I(\beta) & =\sum_{n=0}^{\infty}\beta^{2n}\int_{0<t_{1}<\cdots<t_{n}<\infty}\int_{(M\times M)^{n}}\prod_{k=1}^{n}f(x_{k},y_{k})\nonumber \\
 & \ \ \ \times P_{t_{k}-t_{k-1}}\big((x_{k-1},y_{k-1}),(x_{k},y_{k})\big)d{\bf t}d{\bf x}d{\bf y},\label{eq:IBet1}
\end{align}
where $x_{0}=y_{0}\triangleq x$, $t_{0}\triangleq0$ and 
\[
d{\bf t}\triangleq dt_{1}\cdots dt_{n},\ d{\bf x}d{\bf y}\triangleq\prod_{k=1}^{n}{\rm vol}_{M}(dx_{k}){\rm vol}_{M}(dy_{k}).
\]
Let us further denote $z_{k}\triangleq(x_{k},y_{k})$ to ease notation.
By applying a change of variables $s_{k}\triangleq t_{k}-t_{k-1}$
in (\ref{eq:IBet1}) and noting that each $s_{k}$ has independent integral range
$(0,\infty)$, one finds that 
\begin{align}
I(\beta) & =\sum_{n=0}^{\infty}\beta^{2n}\int_{M^{2}}f(z_{1})\Big(\int_{0}^{\infty}P_{t}(z_{0},z_{1})dt\Big)dz_{1}\times\cdots\times\int_{M^{2}}f(z_{n-1})\nonumber \\
 & \ \ \ \Big(\int_{0}^{\infty}P_{t}(z_{n-2},z_{n-1})dt\Big)dz_{n-1}\int_{M^{2}}f(z_{n})\Big(\int_{0}^{\infty}P_{t}(z_{n-1},z_{n})dt\Big)dz_{n}.\label{eq:IBet2}
\end{align}
Now it becomes clear that the conclusion of Theorem \ref{thm1-intro}-(a)  will follow
directly from the lemma below.
\begin{lem}
\label{lem:KeyLem2MomUp}There exists a positive constant $\Lambda$,
such that 
\begin{equation}
\int_{M\times M}f(x',y')dx'dy'\int_{0}^{\infty}P_{t}((x,y),(x',y'))dt\leqslant\Lambda\label{eq:KeyLem2MomUp}
\end{equation}
for all $(x,y)\in M\times M.$
\end{lem}
\begin{proof}[Proof of Theorem \ref{thm1-intro}-(a)]Presuming that
(\ref{eq:KeyLem2MomUp}) is true, one has from (\ref{eq:2MomRepPf2},
\ref{eq:IBet3}) that 
\[
I(\beta)\leqslant\sum_{n=0}^{\infty}\beta^{2n}\Lambda^{n}.
\]
The conclusion thus follows by taking $\beta_{0}\triangleq1/\sqrt{\Lambda}.$ 

\end{proof}

\subsection{Proof of Lemma \ref{lem:KeyLem2MomUp}}

It remains to prove Lemma \ref{lem:KeyLem2MomUp}. Note that the LHS
of (\ref{eq:KeyLem2MomUp}) equals 
\[
\int_{0}^{\infty}\mathbb{E}_{x,y}\big[f(B_{t},\tilde{B}_{t})\big]dt,
\]
where $B_{t},\tilde{B}_{t}$ are independent Brownian motions starting at $x,y$
respectively. The idea of bounding this integral is very simple; recalling
the assumption (\ref{eq:UpDecayAssump}) on the decay rate of $f$,
one basically replaces $f(B_{t},\tilde{B}_{t})$ by $C\rho(B_{t},\tilde{B}_{t})^{-\alpha}$
and applies the approximation $\rho(B_{t},\tilde{B}_{t})\approx2(d-1)t$
for all large $t$. As we will see below, this is legal since the
heat kernel estimate (\ref{eq:HKEst}) forces the radial process $\rho(x,B_{t})$
to behave like $(d-1)t$ with a $\sqrt{t}$-fluctuation,
while the reversed triangle inequality (\ref{eq:RevTri}) forces $\rho(B_{t},\tilde{B}_{t})\approx\rho(x,B_{t})+\rho(y,\tilde{B}_{t})$ with  high probability. 

We assume without loss of generality that (\ref{eq:UpDecayAssump}) holds for all $x,y\in M$
(by enlarging the constant $C$ if necessary). For each $t>1$, we
will also introduce a localising event $M_{t}$ (its precise definition
is given by (\ref{eq:MainEvt}) below) on which the distance $\rho(B_{t},\tilde{B}_{t})$
can be effectively estimated. One then has the decomposition
\begin{align}
 & \int_{0}^{\infty}\mathbb{E}_{x,y}\big[f(B_{t},\tilde{B}_{t})\big]dt\nonumber \\
 & \leqslant\|f\|_{\infty}+\int_{1}^{\infty}\mathbb{E}_{x,y}\big[f(B_{t},\tilde{B}_{t});M_{t}^{c}\big]dt+\int_{1}^{\infty}\mathbb{E}_{x,y}\big[f(B_{t},\tilde{B}_{t});M_{t}\big]dt\nonumber \\
 & \leqslant\|f\|_{\infty}+\|f\|_{\infty}\int_{1}^{\infty}\mathbb{P}_{x,y}(M_{t}^{c})dt+C\int_{1}^{\infty}\mathbb{E}_{x,y}\big[\rho(B_{t},\tilde{B}_{t})^{-\alpha};M_{t}\big]dt.\label{eq:IntDecomp}
\end{align}
We first analyse the last integral in (\ref{eq:IntDecomp}), which
will also motivate the definition of $M_{t}.$ By applying Lemma \ref{lem:RevTri}
to both triangles $\Delta B_{t}y\tilde{B}_{t}$ and $\Delta B_{t}xy$,
one finds that 
\begin{align}
\rho(B_{t},\tilde{B}_{t}) & \geqslant\rho(B_{t},y)+\rho(y,\tilde{B}_{t})-\log\frac{2}{1-\cos\Psi_{t}}\nonumber \\
 & \geqslant\rho(B_{t},x)+\rho(x,y)-\log\frac{2}{1-\cos\Phi_{t}}+\rho(y,\tilde{B}_{t})-\log\frac{2}{1-\cos\Psi_{t}}\nonumber \\
 & \geqslant\rho(B_{t},x)+\rho(y,\tilde{B}_{t})-\log\frac{2}{1-\cos\Phi_{t}}-\log\frac{2}{1-\cos\Psi_{t}}.\label{eq:LowBdDist1}
\end{align}
where $\Phi_{t}\triangleq\angle B_{t}xy$ and $\Psi_{t}\triangleq\angle B_{t}y\tilde{B}_{t}$
(both $\Phi_{t},\Psi_{t}\in[0,\pi]$),
Let us define the renormalized processes 
\[
\xi_{t}\triangleq\frac{\rho(x,B_{t})-(d-1)t}{\sqrt{t}},\ \eta_{t}\triangleq\frac{\rho(y,\tilde{B}_{t})-(d-1)t}{\sqrt{t}}.
\]
Then (\ref{eq:LowBdDist1}) becomes 
\begin{equation}
\rho(B_{t},\tilde{B}_{t})\geqslant2(d-1)t+\sqrt{t}(\xi_{t}+\eta_{t})-\log\frac{2}{1-\cos\Phi_{t}}-\log\frac{2}{1-\cos\Psi_{t}}.\label{eq:LowBdDist3}
\end{equation}
Let $\delta$ be a fixed positive number such that $4\delta<2(d-1).$
For each $t>1$, we now introduce the earlier announced event:
\begin{equation}
M_{t}\triangleq\Big\{\min\{\xi_{t},\eta_{t}\}>-\delta\sqrt{t},\ \max\big\{\log\frac{2}{1-\cos\Phi_{t}},\log\frac{2}{1-\cos\Psi_{t}}\big\}\leqslant\delta t\Big\}.\label{eq:MainEvt}
\end{equation}
It follows from (\ref{eq:LowBdDist3}) that 
\begin{equation}
\int_{1}^{\infty}\mathbb{E}_{x,y}\big[\rho(B_{t},\tilde{B}_{t})^{-\alpha};M_{t}\big]dt\leqslant\int_{1}^{\infty}\big((2(d-1)-4\delta)t\big)^{-\alpha}<\infty,\label{eq:MainEst}
\end{equation}
since $\alpha>1$ by assumption. 

Next, we estimate $\mathbb{P}_{x,y}(M_{t}^{c})$. This is done by the following
two lemmas (one for estimating $\xi_{t},\eta_{t}$ and the other for $\Phi_{t},\Psi_{t}$). 
\begin{lem}\label{lem:FlucEst}
For any $\delta>0,$ there exists some positive constant $K_{1}$
such that 
\begin{equation}
\mathbb{P}_{x,y}\big(\xi_{t}\leqslant-\delta\sqrt{t}\big)\leqslant e^{-K_{1}t}\label{eq:FlucEst}
\end{equation}
for all $t>1$.
\end{lem}
\begin{proof}
Recall that the volume form under geodesic polar coordinates with
respect to $x$ is ${\rm vol}=\sinh^{d-1}\rho d\rho d\sigma,$ where
$\rho$ is the distance to $x$ and $d\sigma$ is the volume form
for the unit sphere in $T_{x}M$. Given any $I\subseteq\mathbb{R}$,
by the heat kernel estimate (\ref{eq:HKEst}) one has 
\begin{align*}
\mathbb{P}_{x,y}(\xi_{t}\in I) & =\mathbb{P}_{x,y}\big((\rho(x,B_{t})-(d-1)t)/\sqrt{t}\in I\big)\\
 & \leqslant C_{1}t^{-d/2}\int_{\big\{(\rho,\sigma):\frac{\rho-(d-1)t}{\sqrt{t}}\in I\big\}}\exp\Big(-\frac{(d-1)^{2}t}{4}-\frac{\rho^{2}}{4t}-\frac{(d-1)\rho}{2}\Big)\\
 & \ \ \ \times(1+\rho+t)^{\frac{d-3}{2}}(1+\rho)\sinh^{d-1}\rho d\rho d\sigma\\
 & \leqslant C_{2}t^{-d/2}\int_{\big\{(\rho,\sigma):\frac{\rho-(d-1)t}{\sqrt{t}}\in I\big\}}\exp\Big(-\frac{(d-1)^{2}t}{4}-\frac{\rho^{2}}{4t}+\frac{(d-1)\rho}{2}\Big)\\
 & \ \ \ \times(1+\rho+t)^{\frac{d-1}{2}}d\rho,
\end{align*}
where we have used $\sinh\rho\leqslant e^{\rho}/2$ and also integrated
out the angular variable $\sigma.$ By applying a change of variables 
$r=[\rho-(d-1)t]\sqrt{t}$, it is easily seen that 
\[
\exp\Big(-\frac{(d-1)^{2}t}{4}-\frac{\rho^{2}}{4t}+\frac{(d-1)\rho}{2}\Big)=e^{-r^{2}/4}
\]
and (since $t>1$)
\[
t^{-d/2}(1+\rho+t)^{\frac{d-1}{2}}d\rho=t^{-\frac{d-1}{2}}(1+dt+\sqrt{t}r)^{\frac{d-1}{2}}dr\leqslant(1+d+|r|)^{\frac{d-1}{2}}dr.
\]
It follows that 
\begin{equation}
\mathbb{P}_{x,y}(\xi_{t}\in I)\leqslant C_{2}\int_{I}(1+d+|r|)^{\frac{d-1}{2}}e^{-r^{2}/4}dr.\label{eq:FlucTail}
\end{equation}
In our context, one takes $I=(-\infty,-\delta\sqrt{t}].$ It is straight
forward that the right hand side of (\ref{eq:FlucTail}) is bounded above by $e^{-C_{3}t}$
for some positive $C_{3}$ that depends on $\delta.$
\end{proof}
\begin{lem}\label{lem:AngEst}
There exists a universal constant $K_{2}>0$ such that 
\begin{equation}
\mathbb{P}_{x,y}\Big(\log\frac{2}{1-\cos\Phi_{t}}>\delta t\Big)+\mathbb{P}_{x,y}\Big(\log\frac{2}{1-\cos\Psi_{t}}>\delta t\Big)\leqslant K_{2}e^{-\delta t/2}\label{eq:AngEst}
\end{equation}
for all $\delta,t>0.$
\end{lem}
\begin{proof}
Write $B_{t}=(\rho(x,B_{t}),\Theta_{t})$ in geodesic polar coordinates,
where $\Theta_{t}\in S^{1}T_{x}M$ (the unit sphere on $T_{x}M$)
is its angular component. Note that $\mathrm{SO}(d)$ acts on $M$ by isometries
(rotations with respect to $x$ leaving $x$ fixed). As a result,
$\Theta_{t}$ is a uniform random variable on $\mathbb{S}^{d-1}$. In addition,
given any fixed direction $v\in S^{1}T_{x}M$, for the same reason
the angle between $\Theta_{t}$ and $v$ follows a canonical distribution
on $[0,\pi]$ that is independent of $v$ and $t$ (which can of course
be computed explicitly but is not needed here). Let $L$ denote an upper bound of the density
function of this distribution. 

Simple algebra shows that 
\begin{equation}
\log\frac{2}{1-\cos\Phi_{t}}>\delta t\iff\sin\frac{\Phi_{t}}{2}<e^{-\delta t/2}.\label{eq:Angle1}
\end{equation}
Since $\Phi_{t}/2\in[0,\pi/2]$, one can apply Jordan's inequality
(i.e. $\sin\theta\geqslant2\theta/\pi$ for all $\theta\in[0,\pi/2]$) to
see that (\ref{eq:Angle1}) implies $\Phi_{t}<\pi e^{-\delta t/2}.$
But $\Phi_{t}$ is just the angle between $\Theta_{t}$ and the direction
of geodesic $xy$. As a result, one obtains that 
\[
\mathbb{P}_{x,y}\Big(\log\frac{2}{1-\cos\Phi_{t}}>\delta t\Big)\leqslant L\pi e^{-\delta t/2}.
\]
To analyse $\Psi_{t}$, let $\tilde{\Theta}_{t}$ denote the angular
component of $\tilde{B}_{t}$ and let $W_{t}$ be the direction of
the geodesic $yB_{t}.$ Note that $\Psi_{t}$ is the angle between
$\tilde{\Theta}_{t}$ and $W_{t}$. Since $\tilde{\Theta}_{t}$ and
$W_{t}$ are independent (because $B$ and $\tilde{B}$ are), one
also has 
\begin{align*}
\mathbb{P}_{x,y}\Big(\log\frac{2}{1-\cos\Psi_{t}}>\delta t\Big) & =\mathbb{E}_{x,y}\big[\mathbb{P}_{x,y}\big(\log\frac{2}{1-\cos\Psi_{t}}>\delta t\big|W_{t}\big)\big]\\
 & \leqslant\mathbb{E}_{x,y}\big[\mathbb{P}_{x,y}\big(\Psi_t<\pi e^{-\delta t/2}\big|W_{t}\big)\big]\leqslant L\pi e^{-\delta t/2}.
\end{align*}
The desired estimate (\ref{eq:AngEst}) thus follows.
\end{proof}
It follows from (\ref{eq:FlucEst}, \ref{eq:AngEst}) that 
\begin{align*}
\mathbb{P}_{x,y}(M_{t}^{c}) & \leqslant\mathbb{P}_{x,y}(\xi_{t}\leqslant-\delta\sqrt{t})+\mathbb{P}(\eta_{t}\leqslant-\delta\sqrt{t})+\mathbb{P}_{x,y}\Big(\log\frac{2}{1-\cos\Phi_{t}}>\delta t\Big)\\
 & \ \ \ +\mathbb{P}_{x,y}\Big(\log\frac{2}{1-\cos\Psi_{t}}>\delta t\Big)\leqslant2e^{-K_{1}t}+K_{2}e^{-\delta t/2}.
\end{align*}
By substituting this and (\ref{eq:MainEst}) into (\ref{eq:IntDecomp}),
one arrives at 
\begin{align*}
 & \int_{0}^{\infty}\mathbb{E}_{x,y}\big[f(B_{t},\tilde{B}_{t})\big]dt\\
 & \leqslant\|f\|_{\infty}+\int_{1}^{\infty}\Big[\|f\|_{\infty}\big(2e^{-K_{1}t}+K_{2}e^{-\delta t/2}\big)+C\big((2(d-1)-4\delta)t\big)^{-\alpha}\Big]dt.
\end{align*}
This is clearly a finite constant that is independent of the starting
points $x,y$. 

The proof of Lemma \ref{lem:KeyLem2MomUp} is now
complete.

\section{Lower estimates}

In this section, we prove Theorem \ref{thm1-intro}-(b), as well as Theorems \ref{thm2-intro} and \ref{thm3-intro}.

\subsection{Lower bound for large $\beta$}

We first prove Theorem \ref{thm2-intro} which asserts that the growth of $\mathbf{E}[u(t,x)^{2}]$
is always $e^{O(t)}$ when the parameter $\beta$ is large. 

\begin{proof}[Proof of Theorem \ref{thm2-intro}]
We again use the Feynman-Kac representation (\ref{eq:2MomRep}). Let
$r>0$ be chosen such that
\[
m_{f}(r)\triangleq\inf\big\{ f(y,z):\rho(y,x)<r/2,\ \rho(z,x)<r/2\big\}>0,
\]where $x$ is the common starting point of both $B_{t}$ and $\tilde{B}_{t}$. Such an $r$ exists since $f(x,x)>0$ (as  a variance). Define 
\[
\sigma_{r/2}\triangleq\inf\{t:\rho(B_{t},x)>r/2\},\ \tilde{\sigma}_{r/2}\triangleq\inf\{t:\rho(\tilde{B}_{t},x)>r/2\}.
\]
It follows from (\ref{eq:2MomRep}) that 
\begin{align}
\mathbf{E}[u(t,x)^{2}] & \geqslant\mathbb{E}_{x,x}\big[\exp\big(\beta^{2}\int_{0}^{t}f(B_{s},\tilde{B}_{s})ds\big);\sigma_{r/2}>t,\ \tilde{\sigma}_{r/2}>t\big]\nonumber \\
 & \geqslant\exp\big(\beta^{2}t\cdot m_{f}(r)\big)\mathbb{P}_x(\sigma_{r/2}>t)^{2}\label{eq:LarLower1}
\end{align}
for all $t$ and $\beta$. It is well-known that 
\[
\mathbb{P}(\sigma_{r/2}>t)\sim e^{-\lambda_{r/2}t}\ \ \ \text{as }t\rightarrow\infty,
\]
where $\lambda_{\rho}>0$ denotes the principal Dirichlet eigenvalue
of $-\Delta$ on the geodesic ball $B(x,\rho)$. In particular, there
exist constants $C_{r},T_{r}>0$ such that 
\begin{equation}
\mathbb{P}(\sigma_{r/2}>t)\geqslant C_{r}e^{-\lambda_{r/2}t}\ \ \ \forall t>T_{r}.\label{eq:TailExit}
\end{equation}
By substituting (\ref{eq:TailExit}) into (\ref{eq:LarLower1}), one
finds that 
\[
\mathbf{E}[u(t,x)^{2}]\geqslant C_{r}^{2}\exp\big((\beta^{2}m_f(r)-2\lambda_{r/2})t\big)\ \ \ \forall t>T_{r},\beta>0.
\]
This proves the theorem.
\end{proof}

We now discuss an Euclidean result in order to compare it with Theorem~\ref{thm2-intro}.
\begin{proposition}\label{prop:EucLow}
In the Euclidean case $M=\mathbb{R}^{d}$, suppose that $f(x,y)=F(|x-y|)$ where $F$ is a given function
satisfying the following estimate:
\begin{equation}
F(u)\geqslant\frac{C}{(1+u)^{\alpha}}\ \ \ \forall u\geqslant0\label{eq:EucAssumpLower}
\end{equation}
with given constants $C>0$ and $\alpha\in(0,2)$. Then the estimate (\ref{eq:LargeBLower}) is
valid for all $\beta>0.$
\end{proposition}

\begin{proof}
We continue to use the notation introduced in the proof of Theorem \ref{thm2-intro}. In the Euclidean case, the key observation is
that $\lambda_{\rho}=\lambda_{1}/\rho^{2}$ for all $\rho>0$.
On the other hand, by the assumption (\ref{eq:EucAssumpLower}) one
has 
\[
m_{f}(r)\geqslant\frac{C}{(1+r)^{\alpha}}.
\]
It follows from (\ref{eq:LarLower1}) that 
\[
\mathbf{E}[u(t,x)^{2}]\geqslant C_{r}^{2}\exp\Big(\big(\frac{C\beta^{2}}{(1+r)^{\alpha}}-\frac{8\lambda_{1}}{r^{2}}\big)t\Big)
\]for all $t>T_r$ and $\beta>0$.
Given an arbitrary $\beta,$ since $\alpha\in(0,2)$ one can choose
$r$ sufficiently large so that 
\[
\frac{C\beta^{2}}{(1+r)^{\alpha}}-\frac{8\lambda_{1}}{r^{2}}>0.
\]
This implies the growth property (\ref{eq:LargeBLower}) for every
$\beta>0$.
\end{proof}
\begin{rem}
The above argument breaks down in the hyperbolic case since 
\[
\lim_{r\rightarrow\infty}2\lambda_{r/2}=\frac{(d-1)^{2}}{2}>0,\ \lim_{r\rightarrow0^{+}}2\lambda_{r/2}=+\infty.
\]
For an arbitrary $\beta,$ one \textit{cannot} always make $\beta^{2}m_{f}(r)-2\lambda_{r/2}$
positive by suitably choosing $r.$ In fact, one already knows from the upper bound part in
Theorem \ref{thm3-intro} that the conclusion of Proposition \ref{prop:EucLow} is not true 
in the hyperbolic space when $\alpha\in(0,1)$ and $\beta$ is small. 
\end{rem}

\subsection{Lower bound for $\alpha\in(0,1)$}

In this section, we prove the following more quantitative version of Theorem~\ref{thm1-intro}-(b).
\begin{thm}
\label{thm:CrudeLower}Suppose that the covariance function $f$ is
nonnegative and satisfies the following estimate: 
\begin{equation}
f(x,y)\geqslant\frac{C}{\rho(x,y)^{\alpha}}\ \ \ \forall x,y:\rho(x,y)\geqslant R\label{eq:LowDecayCov}
\end{equation}
with suitable constants $C,R>0$ and $\alpha\in(0,1).$ Then there
exists a positive constant $K$ such that 
\begin{equation}\label{eq:CrudeLower}
\mathbf{E}[u^{2}(t,x)]\geqslant e^{K\beta^{2}t^{1-\alpha}}
\end{equation}
for all sufficiently large $t$.
\end{thm}

\subsubsection{The hyperboloid model }

Our proof of Theorem \ref{thm:CrudeLower} relies on a simple localisation
argument which makes explicit use of the hyperboloid model. More specifically,
we define a Lorentzian inner product on $\mathbb{R}^{d+1}$ by 
\[
y*z\triangleq y_{1}z_{1}+\cdots+y_{d}z_{d}-y_{d+1}z_{d+1}
\]
and consider the space 
\[
M\triangleq\{(z_{1},\cdots,z_{d+1})\in\mathbb{R}^{d+1}:z*z=-1,z_{d+1}>0\}.
\]
The Lorentzian inner product induces a Riemannian metric on $M$ by
restriction, which makes $M$ into the space-form of curvature $K\equiv-1$.
The hyperbolic distance between $y,z\in M$ is computed by the formula
\begin{equation}
\cosh\rho(y,z)=-y*z.\label{eq:DistHyperboloid}
\end{equation}
In particular, for the base point $o\triangleq({\bf 0},1)$ one has
\[
\rho(y,o)={\rm arccosh}\,y_{d+1}.
\]
A useful coordinate system on $M$ is the (geodesic) polar coordinates.
Namely, given $\rho>0$ and $\sigma\in \mathbb{S}^{d-1}$ (the unit sphere
on the tangent space $T_{o}M$), the point with polar coordinates
$(\rho,\sigma)$ is given by $y=\exp_{o}(\rho\sigma)$ where $\exp_{o}$
denotes the exponential map at $o$. In other words, $y$ is the location
after travelling along the geodesic from $o$ in the direction $\sigma$
for a distance of $\rho.$ 

We state a simple geometric lemma that will be useful later on. 
\begin{lem}
\label{lem:HypLoc}There exist two subsets $A,B\subseteq \mathbb{S}^{d-1},$
such that
\[
\rho(y,z)\geqslant\max\{\rho(y,o),\rho(z,o)\}
\]
for all $y=(\rho_{1},\sigma_{1}),z=(\rho_{2},\sigma_{2})$ with $\sigma_{1}\in A$
and $\sigma_{2}\in B$.
\end{lem}
\begin{proof}
According to the relation (\ref{eq:DistHyperboloid}), one has 
\begin{equation}
\rho(y,z)\geqslant\rho(y,o)\iff y*z\leqslant y*o\iff\sum_{k=1}^{d}y_{k}z_{k}-y_{d+1}z_{d+1}\leqslant-y_{d+1}.\label{eq:GeoLocal}
\end{equation}
Suppose $A,B$ are two fixed regions on $\mathbb{S}^{d-1}$ (both having positive
spherical volumes) such that
\[
\angle(\sigma,\tau)>\frac{\pi}{2}\ \ \ \forall\sigma\in A,\tau\in B,
\]where $\angle(\sigma,\tau)$ denotes the angle between $O\sigma$ and $O\tau$ ($O$ is the center of $\mathbb{S}^{d-1}$).
Clearly, a choice of $(A,B)$ exists. Given $y=(\rho_{1},\sigma_{1}),z=(\rho_{2},\sigma_{2})$
with $\sigma_{1}\in A$ and $\sigma_{2}\in B$, one has $\sum_{k=1}^{d}y_{k}z_{k}\leqslant0$
and $-y_{d+1}z_{d+1}\leqslant-y_{d+1}$ since $z_{d+1}\geqslant1$.
It follows from (\ref{eq:GeoLocal}) that $\rho(y,z)\geqslant\rho(y,o)$
and by symmetric the same inequality holds with the right hand side replaced by
$\rho(z,o)$.
\end{proof}

\subsubsection{Proof of Theorem \ref{thm:CrudeLower}}

According to Jensen's inequality and the relation (\ref{eq:2MomRep}),
one has 
\begin{equation}
\mathbf{E}[u^{2}(t,x)]\geqslant\exp\Big(\beta^{2}\int_{0}^{t}\mathbb{E}_{x,x}\big[f(B_{s},\tilde{B}_{s})\big]ds\Big).\label{eq:CL3}
\end{equation}
It  suffices to lower bound the integral 
\[
\int_{0}^{t}\mathbb{E}_{x,x}\big[f(B_{s},\tilde{B}_{s})\big]ds.
\]
Recall that the pair $(B_{s},\tilde{B}_{s})$ is a Brownian motion on the product
manifold $M\times M$ (equipped with the product hyperbolic metric),
whose heat kernel is given by 
\begin{equation}
P_{s}((x,y),(z,w))=H_{s}(\rho(x,z))H_{s}(\rho(y,w)).\label{eq:ProdHK}
\end{equation}
It follows from (\ref{eq:ProdHK}) that 
\begin{align}
&\mathbb{E}_{x,x}[f(B_{s},\tilde{B}_{s})]\nonumber\\ & =\int_{M\times M}f(y,z)H_{s}(\rho(x,y))H_{s}(\rho(x,z)){\rm vol}_{M}(dy){\rm vol}_{M}(dz)\nonumber \\
 & \geqslant\int_{\{(y,z):\rho(y,z)>R\}}\frac{C_{1}}{\rho(y,z)^{\alpha}}H_{s}(\rho(x,y))H_{s}(\rho(x,z)){\rm vol}_{M}(dy){\rm vol}_{M}(dz),\label{eq:CL1}
\end{align}
where we also used the assumption (\ref{eq:LowDecayCov}) to reach
the last inequality. 

Now we are going to apply Lemma \ref{lem:HypLoc} in the hyperboloid
model with $x=o$ (the base point). Writing $y=(\rho_{1},\sigma_{1})$
and $z=(\rho_{2},\sigma_{2})$ in polar coordinates with respect to
$o$, the lemma shows that 
\[
\{(y,z):\rho_{1}>R,\rho_{2}>R,\sigma_{1}\in A,\sigma_{2}\in B\}\subseteq\{(y,z):\rho(y,z)>R\}.
\]
It follows from (\ref{eq:CL1}) that 
\begin{align}
\mathbb{E}_{x,x}[f(B_{s},\tilde{B}_{s})] & \geqslant|A|\cdot|B|\int_{R}^{\infty}\int_{R}^{\infty}\frac{C_{1}}{(\rho_{1}+\rho_{2})^{\alpha}}H_{s}(\rho_{1})H_{s}(\rho_{2})\nonumber \\
 & \ \ \ \times\sinh^{d-1}\rho_{1}\sinh^{d-1}\rho_{2}d\rho_{1}d\rho_{2},\label{eq:CL2}
\end{align}
where $|\cdot|$ means taking spherical volume on $S^{d-1}\subseteq T_{o}M$.
To further estimate the right hand side, we localize the double integral (\ref{eq:CL2})
in the region 
\[
\frac{\rho_{i}-(d-1)s}{\sqrt{s}}\in[a,b]\ \ \ (i=1,2),
\]where $a,b$ are two fixed numbers.
It is clear that $\rho_{i}>R$ when $s$ is large. By applying the
hyperbolic heat kernel estimate (\ref{eq:HKEst}) and a change of
variables 
\[
\rho_{i}\longleftrightarrow r_{i}:\rho_{i}=(d-1)s+r_{i}\sqrt{s},
\]
after explicit calculation one finds that the right hand side of (\ref{eq:CL2})
is further bounded from below by 
\[
C_{2}s^{-\alpha}\int_{a}^{b}\int_{a}^{b}e^{-(r_{1}^{2}+r_{2}^{2})/4}dr_{1}dr_{2},
\]
for all large $s$, where $C_{2}$ is a positive constant that is
independent of $s$. To summarise, one concludes that 
\[
\mathbb{E}_{x,x}[f(B_{s},\tilde{B}_{s})]\geqslant C_{3}s^{-\alpha}\ \ \ \forall s>s_{0},
\]
with suitable positive constants $C_{3},s_{0}$. By substituting this
inequality into (\ref{eq:CL3}), one obtains that 
\[
\mathbf{E}[u^{2}(t,x)]\geqslant\exp\Big(C_{3}\beta^{2}\int_{s_{0}}^{t}s^{-\alpha}ds\Big)\geqslant\exp\big(C_{4}\beta^{2}t^{1-\alpha}\big)\ \ \ \forall t>s_{0}.
\]
The proof of Theorem \ref{thm:CrudeLower} is thus complete.

\subsection{Sharpness of (\ref{eq:CrudeLower}) for small $\beta$}

We have seen from Theorem \ref{thm2-intro} that the lower bound (\ref{eq:CrudeLower}) is not sharp when $\beta$ is large (the growth rate is $e^{O(t)}$ in that case). In the following result, we prove that (\ref{eq:CrudeLower}) is optimal for small $\beta$. This result shows that Proposition \ref{prop:EucLow} is not valid for small $\beta$ in the hyperbolic setting, which demonstrates a phenomenon that is not present in the Euclidean situation. Together with Theorem \ref{thm:CrudeLower}, one thus obtains the conclusion of Theorem \ref{thm3-intro}.

\begin{thm}\label{thm:UpperSmallB}
Suppose that the covariance function $f$ satisfies the following
estimate:
\begin{equation}
|f(x,y)|\leqslant\frac{C}{(1+\rho(x,y))^{\alpha}}\ \ \ \forall x,y\in M,\label{eq:AssumpUpper}
\end{equation}
where $C>0$ and $\alpha\in(0,1)$ are given parameters. Then there
exists a positive constant $\bar{\beta}_{0}$ such that 
\[
\sup_{t\geqslant1,x\in M}\frac{1}{t^{1-\alpha}}\log\mathbf{E}[u(t,x)^{2}]<\infty
\]
for every $\beta\in(0,\bar{\beta}_{0})$. 
\end{thm}
\begin{proof}
We may just assume that $f\geqslant0$ (otherwise, replace $f$ by
$|f|$). According to Lemma \ref{lem:2MomRep} (and using the same
notation over there), one can write 
\begin{equation}
\mathbf{E}[u(t,x)^{2}]=\sum_{n=0}^{\infty}\frac{\beta^{2n}}{n!}\int_{[0,t]^{n}}\mathbb{E}_{x,x}\big[f(B_{s_{1}},\tilde{B}_{s_{1}})\cdots f(B_{s_{n}},\tilde{B}_{s_{n}})\big]ds_{1}\cdots ds_{n}.\label{eq:SeriesRep2Mom}
\end{equation}
To estimate the integral on the right hand side, let $\delta\in(0,2(d-1)/3)$
be a fixed number. For each $s>0,$ we introduce the event 
\[
A_{s}\triangleq\Big\{\xi_{s}>-\delta\sqrt{s},\ \eta_{s}>-\delta\sqrt{s},\ \log\frac{2}{1-\cos\angle B_{s}x\tilde{B}_{s}}\leqslant\delta s\Big\},
\]
where 
\[
\xi_{s}\triangleq\frac{\rho(B_{s},x)-(d-1)s}{\sqrt{s}},\ \eta_{s}\triangleq\frac{\rho(\tilde{B}_{s},x)-(d-1)s}{\sqrt{s}}
\]
and $\angle B_{s}x\tilde{B}_{s}\in[0,\pi]$ denotes the intersection
angle between the geodesics $\overline{xB_{s}}$ and $\overline{x\tilde{B}_{s}}$
at $x$. The main idea is to apply the estimate 
\begin{equation}
f(B_{s_{k}},\tilde{B}_{s_{k}})\leqslant\frac{C}{(1+\rho(B_{s_{k}},\tilde{B}_{s_{k}}))^{\alpha}}\leqslant\frac{C}{(1+Ls_{k})^{\alpha}}\label{eq:fBoundLocal}
\end{equation}
on the event $A_{s_{k}}$ where $L\triangleq2(d-1)-3\delta$, which
is a consequence of the assumption (\ref{eq:AssumpUpper}) and the
reverse triangle inequality (\ref{eq:RevTri}). While on the event
$A_{s_{k}}^{c}$, one simply applies $f\leqslant C$ together with
the estimate that 
\begin{equation}
\mathbb{P}_{x,x}(A_{s}^{c})\leqslant K_{1}e^{-K_{2}s}\ \ \ \forall s>0\label{eq:AsComp}
\end{equation}
for suitable constants $K_{1}>1$, $K_{2}>0$; the last inequality
follows from Lemma \ref{lem:FlucEst} as well as a simple adaptation
of Lemma \ref{lem:AngEst}.

To be more specific, one decomposes the expectation on the right hand side of
(\ref{eq:SeriesRep2Mom}) into
\begin{align*}
 & \mathbb{E}_{x,x}\big[f(B_{s_{1}},\tilde{B}_{s_{1}})\cdots f(B_{s_{n}},\tilde{B}_{s_{n}})\big]\\
 & =\sum_{{\cal I}\subseteq\{1,\cdots,n\}}\mathbb{E}_{x,x}\Big[f(B_{s_{1}},\tilde{B}_{s_{1}})\cdots f(B_{s_{n}},\tilde{B}_{s_{n}});\big(\bigcap_{k\in{\cal I}}A_{s_{k}}\big)\cap\big(\bigcap_{l\in{\cal I}^{c}}A_{s_{l}}^{c}\big)\Big].
\end{align*}
According to the inequality (\ref{eq:fBoundLocal}), one has 
\begin{align*}
\mathbb{E}_{x,x}\big[f(B_{s_{1}},\tilde{B}_{s_{1}})\cdots f(B_{s_{n}},\tilde{B}_{s_{n}})\big] & \leqslant C^{n}\sum_{{\cal I}\subseteq\{1,\cdots,n\}}\prod_{k\in{\cal I}}\frac{1}{(1+Ls_{k})^{\alpha}}\mathbb{P}_{x,x}\big(\bigcap_{l\in{\cal I}^{c}}A_{s_{l}}^{c}\big).
\end{align*}
By integrating this inequality over $[0,t]^{n}$, one finds that 
\begin{align}
 & \int_{[0,t]^{n}}\mathbb{E}_{x,x}\big[f(B_{s_{1}},\tilde{B}_{s_{1}})\cdots f(B_{s_{n}},\tilde{B}_{s_{n}})\big]ds_{1}\cdots ds_{n}\nonumber \\
 & \leqslant C^{n}\sum_{{\cal I}\subseteq\{1,\cdots,n\}}\big(\int_{0}^{t}\frac{ds}{(1+Ls)^{\alpha}}\big)^{|{\cal I}|}\int_{[0,t]^{|{\cal I}^{c}|}}\mathbb{P}_{x,x}\big(\bigcap_{l\in{\cal I}^{c}}A_{s_{l}}^{c}\big)\prod_{l\in{\cal I}^{c}}ds_{l}\nonumber \\
 & \leqslant C^{n}\sum_{{\cal I}\subseteq\{1,\cdots,n\}}\big(\frac{(1+Lt)^{1-\alpha}}{L(1-\alpha)}\big)^{|{\cal I}|}\int_{[0,t]^{|{\cal I}^{c}|}}\mathbb{P}_{x,x}\big(A_{\max\{t_{1},\cdots,t_{|{\cal I}^{c}|}\}}^{c}\big)dt_{1}\cdots dt_{|{\cal I}^{c}|}.\label{eq:Up1}
\end{align}
Denoting $l\triangleq|{\cal I}^{c}|$ to ease notation, the last integral
can further be rewritten as 
\[
\int_{[0,t]^{l}}\mathbb{P}_{x,x}\big(A_{\max\{t_{1},\cdots,t_{l}\}}^{c}\big)dt_{1}\cdots dt_{l}=l!\int_{0<t_{1}<\cdots<t_{l}<t}\mathbb{P}_{x,x}(A_{t_{l}}^{c})dt_{1}\cdots dt_{l}.
\]
According to (\ref{eq:AsComp}), the last integral is estimated as
\begin{align}
 & \int_{0<t_{1}<\cdots<t_{l}<t}\mathbb{P}_{x,x}(A_{t_{l}}^{c})dt_{1}\cdots dt_{l}\nonumber \\
 & \leqslant\int_{0<t_{1}<\cdots<t_{l}<t}K_{1}e^{-K_{2}t_{l}}dt_{1}\cdots dt_{l}\nonumber \\
 & =K_{1}\int_{0}^{t}\frac{t_{l}^{l-1}}{(l-1)!}e^{-K_{2}t_{l}}dt_{l}=\frac{K_{1}}{K_{2}^{l}}\int_{0}^{K_{2}t}\frac{r^{l-1}}{(l-1)!}e^{-r}dr\nonumber \\
 & \leqslant\frac{K_{1}}{K_{2}^{l}}\frac{\Gamma(l)}{(l-1)!}=\frac{K_{1}}{K_{2}^{l}}.\label{eq:Up2}
\end{align}
By substiting (\ref{eq:Up2}) back into (\ref{eq:Up1}) and further
into (\ref{eq:SeriesRep2Mom}), one obtains that 
\begin{align}
\mathbf{E}[u(t,x)^{2}] & \leqslant\sum_{n=0}^{\infty}\frac{(C\beta^{2})^{n}}{n!}\sum_{k=0}^{n}{n \choose k}\Big(\frac{(1+Lt)^{1-\alpha}}{L(1-\alpha)}\Big)^{k}\cdot(n-k)!\frac{K_{1}}{K_{2}^{n-k}}\nonumber \\
 & =K_{1}\sum_{n=0}^{\infty}\Big(\frac{C\beta^{2}}{K_{2}}\Big)^{n}\sum_{k=0}^{n}\frac{1}{k!}\Big(\frac{K_{2}(1+Lt)^{1-\alpha}}{L(1-\alpha)}\Big)^{k}\nonumber \\
 & =K_{1}\sum_{k=0}^{\infty}\frac{1}{k!}\Big(\frac{K_{2}(1+Lt)^{1-\alpha}}{L(1-\alpha)}\Big)^{k}\sum_{n=k}^{\infty}\Big(\frac{C\beta^{2}}{K_{2}}\Big)^{n}.\label{eq:Up3}
\end{align}
It is now clear that as long as 
\[
\Lambda_{\beta}\triangleq\frac{C\beta^{2}}{K_{2}}<1\iff\beta<\bar{\beta}_{0}\triangleq\sqrt{\frac{K_{2}}{C}},
\]
the right hand side of (\ref{eq:Up3}) is convergent and is equal to 
\[
\frac{K_{1}}{1-\Lambda_{\beta}}\sum_{k=0}^{\infty}\frac{1}{k!}\Big(\frac{K_{2}\Lambda_{\beta}(1+Lt)^{1-\alpha}}{L(1-\alpha)}\Big)^{k}=\frac{K_{1}}{1-\Lambda_{\beta}}\exp\Big(\frac{K_{2}\Lambda_{\beta}(1+Lt)^{1-\alpha}}{L(1-\alpha)}\Big).
\]
Therefore, one concludes that 
\[
\mathbf{E}[u(t,x)^{2}]\leqslant\frac{K_{1}}{1-\Lambda_{\beta}}\exp\Big(\frac{K_{2}\Lambda_{\beta}(1+Lt)^{1-\alpha}}{L(1-\alpha)}\Big)
\]for all $t>0$,
provided that $\beta<\bar{\beta}_{0}.$ This completes the proof of the
theorem. 
\end{proof}

\section{Covariance functions with power decay }\label{sec: covar-poly}

On the hyperbolic space $M=\mathbb{H}^d$, the covariance function $f(x,y)$ of
a stationary Gaussian field (i.e. a Gaussian field with $\mathrm{SO}^+(d,1)$-invariant distribution) typically decays exponentially
as $\rho(x,y)\rightarrow\infty$. A natural class of such examples
is that $\xi(x)=\int_{\mathrm{SO}^+(d,1)}f(g^{-1}\cdot x)\Xi(dg)$ where $f:M\rightarrow\mathbb{R}$
is a smooth function with suitable decay at infinity and $\Xi$ is
the white noise on the isometry group $\mathrm{SO}^+(d,1)$. Nonetheless, one can construct stationary Gaussian fields
on $M$ whose covariance functions decay with a given power law.
This is demonstrated in Proposition \ref{prop:GauExam} below. 

We
first recall a basic result about non-negative definite kernels (Schoenberg\textquoteright s
theorem).
\begin{defn}
Let $X$ be a topological space. A continuous function $\Psi:X\times X\rightarrow\mathbb{R}$
is \textit{conditionally of negative type} if it satisfies the following
properties:

\vspace{2mm}\noindent (i) $\Psi(x,x)=0$ for all $x\in X$;

\vspace{2mm}\noindent (ii) $\Psi(x,y)=\Psi(y,x)$ for all $x,y\in X$;

\vspace{2mm}\noindent (iii) For any $n\geqslant1$, $x_{1},\cdots,x_{n}\in X$ and any real
numbers $c_{1},\cdots,c_{n}$ with $\sum c_{i}=0$, one has 
\[
\sum_{i,j=1}^{n}c_{i}c_{j}\Psi(x_{i},x_{j})\leqslant0.
\]
A continuous function $\Phi:X\times X\rightarrow\mathbb{C}$ is \textit{of
positive type} if for any $n\geqslant1$, $x_{1},\cdots,x_{n}\in X$
and any complex numbers $c_{1},\cdots,c_{n}$, one has 
\[
\sum_{i,j=1}^{n}c_{i}\overline{c_{j}}\Phi(x_{i},x_{j})\geqslant0.
\]
\end{defn}
\begin{thm}[{Schoenberg's theorem; cf. \cite[Theorem C.3.2]{BHV09}}]\label{thm:Schoenberg}

Let $X$ be a topological space and $\Psi:X\times X\rightarrow\mathbb{R}$
be a continuous function. Then the following statements are equivalent:

\vspace{2mm}\noindent (i) $\Psi$ is conditionally of negative type;

\vspace{2mm}\noindent (ii) The function $e^{-t\Psi}$ is of positive type for every $t\geqslant0$.

\end{thm}

We now give the construction of stationary Gaussian fields on the
hyperbolic space $M$, whose covariance functions admit an exact power
decay. This justifies our earlier assumptions on the covariance function. 
\begin{prop}
\label{prop:GauExam}Let $\alpha>0$ be given fixed. Define $\Psi(x,y)\triangleq\log\cosh\rho(x,y)$.
Then the function 
\begin{equation}
\Phi_{\alpha}(x,y)\triangleq\int_{0}^{1}e^{-u^{1/\alpha}\Psi(x,y)}du,\ \ \ x,y\in M\label{eq:GauExam}
\end{equation}
is of positive type and satisfies 
\[
\lim_{\rho(x,y)\rightarrow\infty}\rho(x,y)^{\alpha}\Phi_{\alpha}(x,y)=\alpha\Gamma(\alpha)
\]
uniformly in $x,y$. In particular, $\Phi_{\alpha}$ defines the covariance
function of a stationary Gaussian field with uniform power decay $\rho(x,y)^{-\alpha}$
as $\rho(x,y)\rightarrow\infty$. 
\end{prop}
\begin{proof}
It is known from \cite[Theorem 2.11.3]{BHV09} that $\Psi$ is conditionally
of negative type. According to Theorem \ref{thm:Schoenberg}, the
function $(x,y)\mapsto e^{-u^{1/\alpha}\Psi(x,y)}$ is of positive
type for every fixed $u\in[0,1]$. In particular, the function $\Phi_{\alpha}(x,y)$
defined by (\ref{eq:GauExam}) is of positive type. To compute its
decay rate, by applying a change of variables $v=u^{1/\alpha}\Psi$
one finds that 
\[
\Phi_{\alpha}(x,y)=\frac{\alpha}{\Psi(x,y)^{\alpha}}\int_{0}^{\Psi(x,y)}e^{-v}v^{\alpha-1}dv.
\]
It follows that 
\[
\rho(x,y)^{\alpha}\Phi_{\alpha}(x,y)=\frac{\alpha\rho(x,y)^{\alpha}}{\Psi(x,y)^{\alpha}}\int_{0}^{\Psi(x,y)}e^{-v}v^{\alpha-1}dv\rightarrow\alpha\Gamma(\alpha)
\]
as $\rho(x,y)\rightarrow\infty$. Since $\Phi_{\alpha}(x,y)$ is a
function of $\rho(x,y)$ only, it is clear that any Gaussian field
with covariance function $\Phi_{\alpha}(x,y)$ is stationary. 
\end{proof}

\end{document}